\RequirePackage{lineno}
\documentclass[12pt,a4paper,leqno]{amsart}
\newcounter{minutes}
\divide\time by 60
\newcounter{hours}
\multiply\time by 60
\addtocounter{minutes}{-\time}
 \def\registered{
 {\ooalign{\hfil\raise .00ex\hbox{\scriptsize R}\hfil\crcr\mathhexbox20D}}}

\usepackage{verbatim}
\usepackage{amssymb}
\usepackage{graphicx}
\usepackage[T1]{fontenc}
\date{}
\newfont{\cyrilic}{wncyr10 scaled 1000}
\newpage
\title{Inequalities for eigenfunctions of the $p$-Laplacian}
\author[B. A. Bhayo]{Barkat Ali Bhayo}
\address{Department of Mathematics, University of Turku,
FI-20014 Turku, Finland} \email{barbha@utu.fi}

\author[M. Vuorinen]{Matti Vuorinen}
\address{Department of Mathematics, University of Turku,
FI-20014 Turku, Finland} \email{vuorinen@utu.fi}

\swapnumbers
\theoremstyle{plain}

\newtheorem{theorem}[equation]{Theorem}
\newtheorem{lemma}[equation]{Lemma}

\newtheorem{remark}[equation]{Remark}

\newtheorem{subsec}[equation]{}
\newtheorem{conjecture}[equation]{Conjecture}

\numberwithin{equation}{section}

\newcommand*\patchAmsMathEnvironmentForLineno[1]{%
  \expandafter\let\csname old#1\expandafter\endcsname\csname #1\endcsname
  \expandafter\let\csname oldend#1\expandafter\endcsname\csname end#1\endcsname
  \renewenvironment{#1}%
     {\linenomath\csname old#1\endcsname}%
     {\csname oldend#1\endcsname\endlinenomath}}%
\newcommand*\patchBothAmsMathEnvironmentsForLineno[1]{%
  \patchAmsMathEnvironmentForLineno{#1}%
  \patchAmsMathEnvironmentForLineno{#1*}}%
\AtBeginDocument{%
\patchBothAmsMathEnvironmentsForLineno{equation}%
\patchBothAmsMathEnvironmentsForLineno{align}%
\patchBothAmsMathEnvironmentsForLineno{flalign}%
\patchBothAmsMathEnvironmentsForLineno{alignat}%
\patchBothAmsMathEnvironmentsForLineno{gather}%
\patchBothAmsMathEnvironmentsForLineno{multline}%
}
\begin{document}
\font\fFt=eusm10 
\font\fFa=eusm7  
\font\fFp=eusm5  
\def\K{\mathchoice
{\hbox{\,\fFt K}}
{\hbox{\,\fFt K}}
{\hbox{\,\fFa K}}
{\hbox{\,\fFp K}}}
\font\fFt=eusm10 
\font\fFa=eusm7  
\font\fFp=eusm5  
\def\E{\mathchoice
{\hbox{\,\fFt E}}
{\hbox{\,\fFt E}}
{\hbox{\,\fFa E}}
{\hbox{\,\fFp E}}}
\maketitle

\begin{abstract}
Motivated by the work of P. Lindqvist, we study eigenfunctions
of the one-dimensional
$p$-Laplace operator, the $\sin_p$ functions, and prove several
inequalities for these and
$p$-analogues of other trigonometric functions and their inverse
functions. Similar inequalities
are given also for the $p$-analogues of the hyperbolic functions and
their inverses.
\end{abstract}


\vspace{.5cm}

\def\thefootnote{}
\footnotetext{
\texttt{\tiny File:~\jobname .tex,
          printed: \number\year-0\number\month-0\number\day,
          \thehours.\ifnum\theminutes<10{0}\fi\theminutes}
}
\makeatletter\def\thefootnote{\@arabic\c@footnote}\makeatother

\vspace{.5cm}

\section{\bf Introduction}


In a highly cited paper P. Lindqvist \cite{lp} studied generalized trigonometric
functions depending on a parameter $p>1$ which for the case $p=2$ reduce to the familiar functions.
Numerous later authors, see e.g. \cite{lpe}, \cite{bem0, bem}, \cite{dm} and
the bibliographies of these papers, have extended this work in various directions
including the study of generalized hyperbolic functions and their inverses.
Our goal here to study these $p$-trigonometric and $p$-hyperbolic functions and to prove several inequalities for them.

For the statement of some of our main results we introduce some
notation and terminology for classical special functions, such as
the classical \emph{gamma function} $\Gamma(x)$, the $\emph
{ psi function}$ $\psi(x)$ and the \emph{beta function} $B(x,y)$.
For ${\rm Re}\, x>0$, ${\rm Re}\, y>0$, these functions are defined by
$$\Gamma(x)=\int^\infty_0 e^{-t}t^{x-1}\,dt,\,\,\psi(x)=\frac{\Gamma^{'}(x)}{\Gamma(x)},\,\,
B(x,y)=\frac{\Gamma(x)\Gamma(y)}{\Gamma(x+y)},$$
respectively.

Given complex numbers $a,b$ and $c$ with $c\neq0,-1,-2,\ldots$,
the \emph{Gaussian hypergeometric function} is the
analytic continuation to the slit place $\mathbf{C}\setminus[1,\infty)$ of the series
$$F(a,b;c;z)={}_2F_1(a,b;c;z)=\sum^\infty_{n=0}\frac{(a,n)(b,n)}
{(c,n)}\frac{z^n}{n!},\qquad |z|<1.$$
Here $(a,0)=1$ for $a\neq 0$, and $(a,n)$ is the \emph{shifted factorial function}
 or the \emph{Appell symbol}
$$(a,n)=a(a+1)(a+2)\cdots(a+n-1)$$
for $n\in\mathbf{N}\setminus \{0\}$, where $\mathbf{N}=\{0,1,2,\ldots\}$.
The hypergeometric function has numerous special functions as its
special or limiting cases, see \cite{AS}.

We start by discussing
eigenfunctions of the so-called one-dimensional
 $p$-Laplacian $\Delta_p$ on $(0,1)$,\,
  $p\in(1,\infty).$ The eigenvalue problem \cite{dm}
$$
-\Delta_p u=-\left(|u^{'}|^{p-2}u^{'}\right)^{'}
=\lambda|u|^{p-2}u,\,u(0)=u(1)=0,
$$
has eigenvalues
$$\lambda_n=(p-1)(n \pi_p)^p,\,$$
and eigenfunctions
$$\sin_p(n \pi_p\, t),\,n\in\mathbb{N},\,$$
where $\sin_p$ is the inverse function of ${\rm arcsin}_p$\,, which is
defined below and
$$\pi_p=\frac{2}{p}\int^1_0(1-s)^{-1/p}s^{1/p-1}ds=\frac{2}
{p}\,B\left(1-\frac{1}{p},\frac{1}{p}\right)=\frac{2 \pi}{p\,\sin(\pi/p)}\,.$$

Motivated by P. Lindqvist's work,
P. J. Bushell and D. E. Edmunds \cite{be} found recently
many new results for these generalized trigonometric functions.
Some authors also considered various other $p$-analogues of trigonometric and hyperbolic
functions and their inverses. In particular, they considered the following
homeomorphisms
\begin{equation*}
\sin_p:(0,a_p)\to I,\,\cos_p:(0,a_p)\to I,\,\tan_p:(0,b_p)\to I,\,\\
\end{equation*}
\begin{equation*}
\sinh_p:(0,c_p)\to I, \, \tanh_p:(0,\infty)\to I\,,
\end{equation*}
where $I=(0,1)$ and
$$a_p=\frac{\pi_p}{2},
\,b_p=\frac{1}{2p}\left(\psi\left(\frac{1+p}{2p}\right)
-\psi\left(\frac{1}{2p}\right)\right)=2^{-1/p}
F\left(\frac{1}{p},\frac{1}{p};1+\frac{1}{p};\frac{1}{2}\right)\,,$$
$$c_p=\left(\frac{1}{2}\right)^{1/p}F\left(1\,,\frac{1}{p}
;1+\frac{1}{p}\,;\frac{1}{2}\right)\,.$$
For $x\in I$, their inverse functions are defined as
\begin{eqnarray*}
{\rm arcsin}_p\,x&=&\int^x_0(1-t^p)^{-1/p}dt=x\,F\left(\frac{1}{p},
\frac{1}{p};1+\frac{1}{p};x^p\right)\\
                 &=&x(1-x^p)^{(p-1)/p}F\left(1,1;1+\frac{1}{p};x^p\right)\,,\\
{\rm arctan}_p\,x&=&\int^x_0(1+t^p)^{-1}dt=x \,
F\left(1,\frac{1}{p};1+\frac{1}{p};-x^p\right)\,,\\
{\rm arsinh}_p\,x&=&\int^x_0(1+t^p)^{-1/p}dt=
xF\left(\frac{1}{p}\,,\frac{1}{p};1+\frac{1}{p};-x^p\right)\,,\\
{\rm artanh}_p\,x&=&\int^x_0(1-t^p)^{-1}dt=xF\left
(1\,,\frac{1}{p};1+\frac{1}{p};x^p\right)\,,
\end{eqnarray*}
and  by \cite[Prop 2.2]{be} ${\rm arccos}_p\,x={\rm arcsin}_p((1-x^p)^{1/p})$.
For the particular case $p=2$ one obtains the familiar elementary functions.

The paper is organized into sections as follows. Section 1, the introduction,
contains the statements of our main results. In Section 2 we give some
inequalities for the $p$-analogues of trigonometric and hyperbolic functions.
Section 3 contains the proofs of our main results and some identities. Finally
in Section 4 we give some functional inequalities for elementary functions and
Section 5 contains two small tables with a few values of the function $\sin_p$ and related functions compiled with the Mathematica$^\registered$ software.

Some of the main results are the following theorems.

\begin{theorem}\label{d} For $p>1$ and $x\in(0,1)$, we have
\begin{enumerate}
\item $\left(1+\displaystyle\frac{x^p}{p(1+p)}\right)x<
{\rm arcsin}_p\,x< \displaystyle\frac{\pi_p}{2}\,x$,
\item $\left(1+\displaystyle\frac{1-x^p}{p(1+p)}\right)(1-x^p)^{1/p}
< {\rm arccos}_p\,x< \displaystyle\frac{\pi_p}{2}\,
(1-x^p)^{1/p}$,
\item $\displaystyle\frac{(p(1+p)(1+x^p)+x^p)x}{p(1+p)(1+x^p)^{1+1/p}}<
{\rm arctan}_p\,x
< 2^{1/p}\,\,b_p\,
\left(\displaystyle\frac{x^p}{1+x^p}\right)^{1/p}$.
\end{enumerate}
\end{theorem}

\begin{theorem}\label{thm1} For $p>1$ and $x\in(0,1)$, we have
\begin{equation}\label{thm2}
z\left (1+\frac{\log(1+x^p)}{1+p}\right)
< {\rm arsinh}_p\,x
< z\left(1+\frac{1}{p}\log(1+x^p)\right),
\,\,z=\displaystyle\left(\frac{x^p}{1+x^p}\right)^{1/p},
\end{equation}

\begin{equation}\label{thm3}
x\left(1-\frac{1}{1+p}\log(1-x^p)\right)< {\rm artanh}_p\,x
< x\left(1-\frac{1}{p}\log(1-x^p)\right)\,.
\end{equation}
\end{theorem}

The next result provides several families of inequalities for elementary
functions.

\begin{theorem}\label{thm6}  For $x >0$ and $z=\pi x/2$, the function
$g(p)=f(z^p)^{1/p}$
is decreasing in $ p\in(0,\infty)$, where
 $f(z)\in\{{\rm arsinh}(z),{\rm arcosh}(z),{\rm artanh}(2 z/\pi)\}$.
\end{theorem}

{\sc Acknowledgments.} The first author is indebted to the Graduate
School of Mathematical Analysis and its Applications for support.
The second author was, in part supported by the Academy of Finland, Project
2600066611. Both authors wish to acknowledge the expert help of
Dr. H. Ruskeep\"a\"a in the use of the Mathematica$^\registered$ software \cite{ru}
and Prof. P. H\"ast\"o for providing simplified versions of some of our proofs.

\section{\bf Preliminaries and definitions}
\medskip

For convenience, we use the notation $\mathbb{R}_+=(0,\infty)\,.$

\medskip

\begin{lemma}\label{neu}\cite[Thm 2.1]{ne2} Let $f:\mathbb{R}_+\to \mathbb{R}_+$
be a differentiable, log-convex function and let $a\geq 1$. Then $g(x)=(f(x))^a/f(a\,x)$
 decreases on its domain. In particular, if $0\leq x\leq y\,,$ then the following inequalities
 $$\frac{(f(y))^a}{f(a\,y)}\leq\frac{(f(x))^a}{f(a\,x)}\leq (f(0))^{a-1}$$
 hold true. If $0<a\leq 1$, then the function $g$ is an increasing function on $\mathbb{R}_+$
and inequalities are reversed.
\end{lemma}

For easy reference we recall the following identity  \cite[15.3.5]{AS}
\begin{equation}\label{as}
F(a,b;c;z)=(1-z)^{-b}F(b,c-a;c;-z/(1-z))\, .
\end{equation}

For the following lemma see \cite[Thms 1.19(10), 1.52(1), Lems, 1.33, 1.35]{avvb}.

\begin{lemma}\label{avb} \begin{enumerate}
\item For $a,b,c>0$, $c<a+b$, and $|x|<1$,
$$F(a,b;c;x)=(1-x)^{c-a-b}F(c-a,c-b;c;x)\,.$$
\item For $a,x\in(0,1)$, and  $b,c\in(0,\infty)$
$$F(-a,b;c;x)<1-\frac{a\,b}{c}\,x\,.$$
\item For $a,x\in(0,1)$, and  $b,c\in(0,\infty)$
$$F(a,b;c;x)+F(-a,b;c;x)>2\,.$$
\item Let $a,b,c\in(0,\infty)$ and $c>a+b$. Then for $x\in[0,1]$,
$$F(a,b;c;x)\leq \frac{\Gamma(c)\Gamma(c-a-b)}{\Gamma(c-a)\Gamma(c-b)}\,.$$
\item For $a,b>0$, the following function
$$f(x)=\frac{F(a,b;a+b;x)-1}{\log(1/(1-x))}$$
is strictly increasing from $(0,1)$ onto $(a\,b/(a+b),1/B(a,b))$.
\end{enumerate}
\end{lemma}

\begin{lemma}\label{pf} For $p>1$ and $x\in(0,1)$, the functions
$$({\rm arcsin}_p(x^k))^{1/k}\quad {\rm and}\quad ({\rm artanh}_p(x^k))^{1/k}$$
are decreasing in $k\in(0,\infty)$, also
$$({\rm arctan}_p(x^k))^{1/k}\quad {\rm and}\quad ({\rm arsinh}_p(x^k))^{1/k}$$
are increasing in $k\in(0,\infty)$.

In particular, for $k\geq 1$
$$\sqrt[k]{{\rm arcsin}_p(x^k)}\leq {\rm arcsin}_p(x)
\leq ({\rm arcsin}_p\sqrt[k]{x})^k\,,$$
$$\sqrt[k]{{\rm artanh}_p(x^k)}\leq {\rm artanh}_p(x)
\leq ({\rm artanh}_p\sqrt[k]{x})^k\,.$$
$$({\rm arsinh}_p\sqrt[k]{x})^k\leq {\rm arsinh}_p(x)
\leq \sqrt[k]{{\rm arsinh}_p(x^k)}\,,$$
$$({\rm arctan}_p\sqrt[k]{x})^k\leq {\rm arctan}_p(x)
\leq \sqrt[k]{{\rm arctan}_p(x^k)}\,.$$

\end{lemma}

\begin{proof} Let
 let
\[
f(k) := \big( E(x^k)\big)^{1/k}, E(x) := \int_0^x g(t)\, dt, E=E(x^k).
\]
We get
\[
f' = -E^{1/k} \log E \frac1{k^2} + \frac1k E^{1/k -1} E' x^k \log x
= \frac{E^{1/k}}{k^2} \Big( -\log \frac{E}{x^k}  - \big(x^k  \frac{E'}{E} - 1 \big) \log \frac1{x^{k}}\big).
\]
If $g\ge 1$, then
\[
\frac{E}{x^k} = \frac{1}{x^k}\int_0^{x^k} g(t)\, dt \ge 1.
\]
If $g$ is increasing, then
\[
E' - \frac{E}{x^k} = g(x^k) - \frac{1}{x^k}\int_0^{x^k} g(t)\, dt \ge 0,
\]
so that $x^k  \frac{E'}{E} - 1 \ge 0$. Thus $f'\le 0$ under these assumptions.

For arcsin$_p$ and artanh$_p$, $g$ is $(1-t^p)^{-1/p}$ and $(1-t^p)^{-1}$, so the
conditions are clearly satisfied. Additionally, we see that for arsinh$_p$ and arctan$_p$
the conditions $g\le 1$ and $g$ is decreasing and this conclude that $f'\ge 0$. This completes the proof.

\end{proof}

%

\begin{theorem}\label{pf1} For $p>1$ and $r,s\in(0,1)$, the following inequalities hold
\\
\begin{enumerate}
\item ${\rm arcsin}_p(r\,s)\leq \sqrt{{\rm arcsin}_p(r^2)\,{\rm arcsin}_p(s^2)}\leq
{\rm arcsin}_p(r)\,{\rm arcsin}_p(s),\,\, r,s\in(0,1)\,,$\\

\item ${\rm artanh}_p(r\,s)\leq \sqrt{{\rm artanh}_p(r^2)\,{\rm artanh}_p(s^2)}\leq
{\rm artanh}_p(r)\,{\rm artanh}_p(s),\,\, r,s\in(0,1)\,,$

\item ${\rm arsinh}_p(r^2)\,{\rm arsinh}_p(s^2)\leq\sqrt{{\rm arsinh}_p(r^2)\,{\rm arsinh}_p(s^2)}\leq
{\rm arsinh}_p(r\,s)\,,$

\item ${\rm artanh}_p(r)\,{\rm artanh}_p(s)\leq \sqrt{{\rm artanh}_p(r^2)\,{\rm artanh}_p(s^2)}\leq{\rm artanh}_p(r\,s)\,.$

\end{enumerate}
\end{theorem}
\begin{proof}
Let $h(x) := \log f(e^x)$. Then $h$ is convex (in the $C^2$ case)
when $h''\ge 0$, i.e.\ iff
\[
\frac f y (f' + y f'') \ge (f')^2,
\]
where $y=e^x$ and the function is evaluated at $y$. If $f''\ge 0$, then
$$\frac f y \ge f'(0)$$
, so a sufficient condition for convexity is
$f'(0) (f' + y f'') \ge (f')^2$. If $f''\le 0$, the reverse holds, so a
sufficient condition for concavity is $f'(0) (f' + y f'') \le (f')^2$.
Suppose
\[
f(x) := \int_0^x g(t)\, dt.
\]
Then $f' = g$ and $f'' = g'$. Then one easily checks that
$h$ is convex in case $g$ is $(1-t^p)^{-1/p}$ and $(1-t^p)^{-1}$,
and concave for $g$ equal to $(1+t^p)^{-1/p}$ and $(1+t^p)^{-1}$.¨
Now proof follows easily from Lemma \ref{pf}.
\end{proof}

\begin{lemma} For $k,p>1$ and $r \geq s$, we have
\begin{eqnarray*}
\left(\frac{{\rm arcsin}_p(s)}{{\rm arcsin}_p(r)}\right)^{k}&\leq&
\frac{{\rm arcsin}_p(s^k)}{{\rm arcsin}_p(r^k)}\,,\,\,\, r,s\in(0,1),\\
\left(\frac{{\rm artanh}_p(s)}{{\rm artanh}_p(r)}\right)^{k}&\leq&
\frac{{\rm artanh}_p(s^k)}{{\rm artanh}_p(r^k)}\,,\,\,\, r,s\in(0,1),\\
\frac{{\rm arsinh}_p(s^k)}{{\rm arsinh}_p(r^k)}&\leq&
\left(\frac{{\rm arsinh}_p(s)}{{\rm arsinh}_p(r)}\right)^{k}
\,,\,\,\, r,s\in(0,1)\,.
\end{eqnarray*}
\end{lemma}

\begin{proof} For $x>0$, the following functions
$$u(x)={\rm arcsin}_p(e^{-x})\,,\quad v(x)={\rm artanh}_p(e^{-x})\,,$$
$$w_1(x)=1/{\rm arsinh}_p(e^{-x})$$  
are log-convex by the proof of Theorem \ref{pf1}. Let $x<y$, $e^{-x}=r\geq s=e^{-y}$, now
inequalities follow from Lemma \ref{neu}.
\end{proof}

\begin{lemma}\label{ku}\cite[Thm 2, p.151]{ku}
Let $J\subset\mathbb{R}$ be an open interval, and let $f:J\to \mathbb{R}$
be strictly monotonic function. Let $f^{-1}:f(J)\to J$ be the inverse to $f$ then
\begin{enumerate}
\item if $f$ is convex and increasing, then $f^{-1}$ is concave,
\item if $f$ is convex and decreasing, then $f^{-1}$ is convex,
\item if $f$ is concave and increasing, then $f^{-1}$ is convex,
\item if $f$ is concave and decreasing, then $f^{-1}$ is concave.
\end{enumerate}
\end{lemma}

\begin{lemma} For $k,p>1$ and $r\geq s$, we have
\begin{eqnarray*}
\left(\frac{\sin_p(r)}{\sin_p(s)}\right)^k&\leq& \frac{\sin_p(r^k)}{\sin_p(s^k)}
,\,\, r,s\in(0,1),\\
\left(\frac{\tanh_p(r)}{\tanh_p(s)}\right)^k&\leq& \frac{\tanh_p(r^k)}{\tanh_p(s^k)}
,\,\, r,s\in(0,\infty),\\
\left(\frac{\sinh_p(r)}{\sinh_p(s)}\right)^k&\geq& \frac{\sinh_p(r^k)}{\sinh_p(s^k)}
,\,\, r,s\in(0,1).
\end{eqnarray*}
Inequalities reverse for $k\in(0,1)$.
\end{lemma}

\begin{proof} It is clear from the proof of Theorem \ref{pf1} that the functions
$$f(x)=\log({\rm arcsin_p}(e^{-x})),\,g(x)=\log({\rm artanh_p}(e^{-x}))
,\, h(x)=\log(1/{\rm arsinh_p}(e^{x}))$$
are convex and decreasing, then Lemma \ref{ku}(2) implies that
$$f^{-1}(y)=\log(1/\sin_p(e^y)),\,g^{-1}(y)=\log(1/\tanh_p(e^y))
,\,h^{-1}(y)=\log(\sinh_p(e^{-y})),\,$$
are convex, now the result follows from Lemma \ref{neu}.
\end{proof}

\begin{lemma}\label{nr} For $p>1$, the following inequalities hold
\begin{enumerate}
\item $\sqrt{\sin_p(r^2)\sin_p(s^2)}\leq \sin_p(r\,s)\,,\, r,s\in(0,\pi_p/2)\,,$\\

\item $\sqrt{\tanh_p(r^2)\tanh_p(s^2)}\leq \tanh_p(r\,s)\,,\, r,s\in(0,\infty)\,,$\\

\item $\sinh_p(r\,s)\leq\sqrt{\sinh_p(r^2)\sinh_p(s^2)}
\,,\, r,s\in(0,\infty)\,.$
\end{enumerate}
\end{lemma}

\begin{proof} Let $f(z)=\log({\rm arcsin}_p(e^{-z})),\,z>0$. Then
$$f^{'}(z)=-(1-e^{-pz})^{-1/p}/F(1/p,1/p;1+1/p;e^{-pz})<0,$$
$f$ is decreasing and by the proof of Theorem \ref{pf1}
$f$ is convex. By Lemma \ref{ku}(2),
$f^{-1}(y)=\log(1/\sin_p(e^y))$ is convex. This implies that
$$\log\left(\frac{1}{\sin_p(e^{x/2}e^{y/2})}\right)\leq \frac{1}
{2}\left(\log\left(\frac{1}{\sin_p(e^{x})}\right)
+\log\left(\frac{1}{\sin_p(e^{y})}\right)\right)\,,$$
letting $r=e^{x/2}$ and $s=e^{y/2}$, we get the first inequality.

For (2), let
$g(z)=\log({\rm artanh}_p(e^{-z})),\,z>0$ and
$$g^{'}(z)=-1/((1-e^{-pz})F(1,1/p;1+1/p;e^{-pz}))<0,$$
hence $g$ is decreasing and by Theorem \ref{pf1} $g$ is convex.
Then $g^{-1}(y)=\log(1/{\rm artanh}_p(e^y))$ is convex by Lemma \ref{ku}(2),
and (2) follows. Finally, let
$h_1(z)=\log({1/\rm arsinh}_p(e^{z}))$ and
$$h_1^{'}(z)=-1/F\left(1,1/p;1+1/p;\frac{e^{pz}}
{1+e^{pz}}\right)<0.$$
Then $h_1^{-1}(y)=\log(\sinh_p(e^{-y}))$ is decreasing and
convex by Lemma \ref{ku}(2).
This implies that
$$\log(1/\sinh_p(e^{-x/2}e^{-y/2}))\leq (\log(1/\sinh_p(e^{-x}))+
\log(1/\sinh_p(e^{-y})))/2,$$
and (3) holds for $r,s\in(0,\infty)$. Again $h_2(z)=\log(1/{\rm arsinh}_p(e^{-z}))$ and
$$h_2^{'}(z)=(F(1,1/p;1+1/p;1/(1+e^{pz})))^{-1}>0,$$
similarly proof follows from Lemma \ref{ku}(2), this completes the proof of (3).
\end{proof}

\begin{lemma}  For $p>1$, the following relations hold
\begin{enumerate}
\item $\sqrt{\sin_p(r)\sin_p(s)}\leq \sin_p((r+s)/2),\,\,r,s\in (0,\pi_p/2),$
\item $\sqrt{\sinh_p(r)\sinh_p(s)}\leq \sinh_p((r+s)/2),\,\,r,s \in(0,\infty)\,.$
\end{enumerate}
\end{lemma}

\begin{proof} The proof follows easily from Lemma \ref{nr} and $2\sqrt{r\,s}\leq r+s$
since the functions are increasing.

\end{proof}

\begin{lemma} For $p>1$, the following inequalities hold
\begin{enumerate}
\item $\sin_p(r+s)\leq \sin_p(r)+\sin_p(s)\,,\,\, r, s\in(0,\pi_p/4)\,,$
\item $\tanh_p(r+s)\leq \tanh_p(r)+\tanh_p(s)\,,\,\, r, s\in(0,b_p/2) \,,$
\item $\tan_p(r+s)\geq \tan_p(r)+\tan_p(s)\,,\,\, r, s\in(0,b_p/2)\,,$
\item $\sinh_p(r+s)\geq \sinh_p(r)+\sinh_p(s)\,,\,\, r, s\in(0,c_p/2)\,$.
\end{enumerate}
\end{lemma}

\begin{proof} Let $f(x)={\rm arcsin}_p(x)$,\,
$x\in(0,a_p)$. We get
$$f^{'}(x)=(1-x^p)^{-1/p}\,,$$
which is increasing, hence $f$ is convex.
Clearly, $f$ is increasing. Therefore
$$f_1=f^{-1}(y)=\sin_p(y)$$
is concave by Lemma \ref{ku}(1). This implies that
$f_1^{'}$ is decreasing. Clearly $f_1(0)=0$, and by
 \cite[Thm 1.25]{avvb}, $f_1(y)/y$
is decreasing. Now it follows from \cite[Lem 1.24]{avvb}
that
$$f_1(r+s)\leq f_1(r)+f_1(s),$$
and (1) follows. The proofs of the remaining claims follow similarly.
\end{proof}
\section{\bf Proof of main results}

\begin{subsec}{\bf Proof of Theorem \ref{d}.} \rm
By Lemma \ref{avb}(3), (2) we get
$$2-\left(1-\displaystyle\frac{x^p}{p\,(1+p)}\right)<
F\left(\frac{1}{p},\frac{1}{p};1+\frac{1}{p};x^p\right)\,,$$
and the first inequality of part one holds.
For the second one we get
\begin{eqnarray*}
{\rm arcsin}_p\,x&=&x\,F\left(\frac{1}{p},\frac{1}{p};1+\frac{1}{p};x^p\right)\\
               &<&\frac{x\,\Gamma(1+1/p)\Gamma(1+1/p-1/p-1/p)}
               {\Gamma(1+1/p-1/p)\Gamma(1+1/p-1/p)}\\
               &=&x\,\Gamma\left(1+\frac{1}{p}\right)\Gamma\left(1-\frac{1}{p}\right)
               =x\,\frac{1}{p}B\left(1-\frac{1}{p},\frac{1}{p}\right)=x\,\frac{\pi_p}{2}
\end{eqnarray*}
by Lemma \ref{avb}(4).
By \cite[Prop (2.11)]{be}, ${\rm arccos}_p\,x={\rm arcsin}_p\,((1-x^p)^{1/p})$,
 and (2)
follows from (1). For (3), if we replace $b=1,c-a=1/p,\,c=1+1/p,\,x^p=z/(1-z)$
in (\ref{as}) then we get
\begin{eqnarray*}
{\rm arctan}_p\,x&=&x\,F\left(1,\frac{1}{p};1+\frac{1}{p};-x^p\right)\\
&=&\left(\frac{x}{1+x^p}\right)F
\left(1,1;1+\frac{1}{p};\frac{x^p}{1+x^p}\right)\\
&=& \left(\frac{x}{1+x^p}\right)\left(\frac{1}{1+x^p}\right)^{1/p-1}
F\left(\frac{1}{p},\frac{1}{p};1+\frac{1}{p};\frac{x^p}{1+x^p}\right)\\
&=& \left(\frac{x^p}{1+x^p}\right)^{1/p}
F\left(\frac{1}{p},\frac{1}{p};1+\frac{1}{p};\frac{x^p}{1+x^p}\right)\\
&< & 2^{1/p}\,b_p\,\left(\frac{x^p}{1+x^p}\right)^{1/p},
\end{eqnarray*}
third identily and inequality follow from Lemma \ref{avb}(1), (4). For the lower bound we get
\begin{eqnarray*}
{\rm arctan}_p\,x&> & \left(\frac{x^p}{1+x^p}\right)^{1/p}
\left(2-F\left(\frac{1}{p},\frac{1}{p};1+\frac{1}{p};
\frac{x^p}{1+x^p}\right)\right)\\
&> & \displaystyle\frac{(p(1+p)(1+x^p)+x^p)x}{p(1+p)(1+x^p)^{1+1/p}}
\end{eqnarray*}
from Lemma \ref{avb}(3),(2).
\end{subsec}

\begin{subsec}{\bf Proof of Theorem \ref{thm1}.} \rm
For (\ref{thm2}),
we replace $b=1/p,\,c-a=1/p,\,c=1+1/p$ and $x^p=z/(1-z)$ in (\ref{as}) and see that
$${\rm arsinh}_p\,x=x\,F\left(\frac{1}{p}\,,\frac{1}{p}\,;1+\frac{1}{p}\,;-x^p\right)
=\left(\frac{x^p}{1+x^p}\right)^{1/p}
F\left(1,\,\frac{1}{p};\,1+\frac{1}{p};\,\frac{x^p}{1+x^p}\right).$$
Now we get
$$\frac{\log\left(1+x^p\right)}{1+p}\left(\frac{x^p}{1+x^p}\right)^{1/p}< $$
$$x \,F\left(\frac{1}{p}\,,\frac{1}{p}\,;1+\frac{1}{p}\,;-x^p\right)< \left(1-\frac{1}{p}\log\left(1-\frac{x^p}{1+x^p}\right)\right)
\left(\frac{x^p}{1+x^p}\right)^{1/p}$$
from Lemma \ref{avb}(5) and observing that $B(1,1/p)=p,$
this implies (\ref{thm2}).

For (\ref{thm3}) we get from Lemma \ref{avb}(5)
$$\frac{1}{1+p}\log\left(\frac{1}{1-x^p}\right)+1< F\left(1\,,\frac{1}{p}
\,;1+\frac{1}{p}\,;x^p\right)<
\frac{1}{p}\log\left(\frac{1}{1-x^p}\right)+1\,,$$
which is equivalent to
$$x\left(1-\frac{1}{1+p}\log(1-x^p)\right)< x\,
F\left(1\,,\frac{1}{p}\,;1+\frac{1}{p}\,;x^p\right)<
x\left(1-\frac{1}{p}\log(1-x^p)\right)\,,$$
and the result follows.
\end{subsec}
\begin{remark} \rm  For the particular case $p=2$. Zhu \cite{zhu1} has proved for $x>0$
$$\frac{6\sqrt{2}(\sqrt{1+x^2}-1)^{1/2}}{4+\sqrt{2}(\sqrt{1+x^2}+1)^{1/2}}< {\rm arsinh}(x).$$
When $p=2$, our bound in Theorem \ref{thm1}(1) differs from this bound roughly $0.01$
when $x\in(0,1)$.
\end{remark}

\begin{lemma} For $p>1$ and $x\in(0,1)$, the following inequalities hold:\\
\begin{enumerate}
\item ${\rm arctan}_p(x)<{\rm arsinh}_p(x)<{\rm arcsin}_p(x)<{\rm artanh}_p(x)\,,$\\
\item $\tanh_p(z)<\sin_p(z)<\sinh_p(z)<\tan_p(z)\,,$\\
\end{enumerate}
the first and the second inequalities hold for $z\in(0,\pi_p/2)$, and the
third one holds for $z\in(0,b_p)$.
\end{lemma}

\begin{proof} From the definition of the $p$-analogues functions we get (1), and (2)
follows from (1).
\end{proof}


\begin{lemma} For $p>1$, we have
$$\frac{6p^2}{3p^2-2}\leq \pi_p\leq \frac{12 p^2}{6p^2-\pi^2}\,,\,\,\pi_p=\frac{2\pi}{p\,\sin(\pi/p)}\,.$$
\end{lemma}

\begin{proof} By \cite[Thm 3.1]{kvv} we get
$$\frac{\pi}{p}\left(1-\frac{\pi^2}{6p^2}\right)\leq \sin\left(\frac{\pi}{p}\right)
\leq \frac{\pi}{p}\left(1-\frac{2}{3p^2}\right)\,,$$
and the result follows easily.
\end{proof}

%
%

\begin{lemma}  For $a \in (0,1)$ and $k,r,s\in(1,\infty)\,,$ the following inequalities hold
\begin{enumerate}
\item $\displaystyle\pi_{r\,s}\leq \displaystyle\sqrt{\pi_{r^2}\,\pi_{s^2}}\leq \displaystyle\sqrt{\pi_{r}\,\pi_{s}} \,,$
\item $\pi_{r^a\,s^{1-a}}\leq a\,\pi_{r}+(1-a)\pi_s \,,$
\item $\displaystyle\left(\frac{\pi_s}{\pi_r}\right)^k\leq \displaystyle\frac{\pi_{s^k}}{\pi_{r^k}}$, $\,r\leq s \,.$
\end{enumerate}
\end{lemma}

\begin{proof} Let $f(x)=\log(\pi_{e^x}),\,x>0$. We get
$$f^{''}(x)=e^{-2 x} \pi ^2 (\csc \left(e^{-x} \pi \right))^2-e^{-x} \pi
   \cot \left(e^{-x} \pi \right)\,,$$
which is positive, because the function $g(y)=y^2(\csc(y))^2-y\cot(y) $ is positive. This implies that
$f$ is convex. Hence

$$\log(\pi_{e^{(x+y)/2}}) \leq \frac{1}{2}\left(\log (\pi_{e^x})+\log(\pi_{e^y})\right)\,,$$
setting $r=e^{x/2}$ and $r=e^{y/2}$, we get the first inequality of (1), and the second one follows from the
fact that $\pi_p$ is decreasing in $p\in(1,\infty)$. Now it is clear that $\pi_{e^x}$ is convex, and we get
$$\pi_{e^{a\,x+(1-a)y}}\leq a\,\pi_{e^x}+(1-a)\pi_{e^y}\,,$$
and (2) follows easily. Let $0\leq x\leq y$, then we get
$$\frac{(\pi_{e^y})^k}{\pi_{e^{k\,y}}}\leq \frac{(\pi_{e^x})^k}{\pi_{e^{k\,x}}}$$
from Lemma \ref{neu}, and (3) follows if we set $r=e^{x}$ and $r=e^{y}$.
\end{proof}

\begin{lemma}
For $p>1$ and $x\in(0,1)$, we have
$${\rm arcsin}_p\left(\frac{x}{\sqrt[p]{1+x^p}}\right)={\rm arctan}_p(x)\,,$$
$${\rm arcsin}_p (x)={\rm arctan}_p\left(\frac{x}{\sqrt[p]{1-x^p}}\right)\,,$$
$${\rm arccos}_p( x)={\rm arctan}_p\left(\frac{\sqrt[p]{1-x^p}}{x}\right)\,,$$
$${\rm arccos}_p\left(\frac{1}{\sqrt[p]{1+x^p}}\right)={\rm arctan}_p (x)\,.$$
\end{lemma}

\begin{proof}
We get
\begin{eqnarray*}
{\rm arctan}_p(x)&=& x \,F\left(\frac{1}{p},\frac{1}{p};1+\frac{1}{p};-x^p\right)\\
             &=&\frac{x}{1+x^p}\,F\left(1,1;1+\frac{1}{p};\frac{x^p}{1+x^p}\right)\\
             &=&\frac{x}{1+x^p}\left(\frac{1}{1+x^p}\right)^{1/p-1}F\left
             (\frac{1}{p},\frac{1}{p}
             ;1+\frac{1}{p};\frac{x^p}{1+x^p}\right)\\
             &=&\left(\frac{x}{1+x^p}\right)^{1/p}F\left(\frac{1}{p},\frac{1}{p};1+\frac{1}{p}
             ;\left(\frac{x}{(1+x^p)^{1/p}}\right)^p\right)\\
             &=&{\rm arcsin}_p\left(\frac{x}{\sqrt[p]{1+x^p}}\right)
\end{eqnarray*}
by (\ref{as}) and Lemma \ref{avb}(1). Write $y=x/\sqrt[p]{1-x^p}$, and second follows from first one.
 For the third identity, we get
\begin{eqnarray*}
{\rm arctan}_p\left(\frac{\sqrt[p]{1-x^p}}{x}\right)&=& {x} \sqrt[p]{1-x^p}
F\left(\frac{1}{p},\frac{1}{p};1+\frac{1}{p};(1-x^p)\right)\\
&=& {\rm arcsin}_p((1-x^p)^{1/p})={\rm arccos}_p(x)
\end{eqnarray*}
by (\ref{as}), Lemma \ref{avb}(1) and \cite[Prop 2.2]{be}. Similarly, the fourth identity follows from
third one.
\end{proof}

\begin{conjecture} For a fixed $x\in(0,1)$, the functions
$${\rm sin}_p(\pi_p \,x/2),\,{\rm tan}_p(\pi_p\, x/2),\, \sinh_p(c_p\,x)$$
are monotone in $p\in(1,\infty)$. For fixed $x>0$, $\tanh_p(x)$ is increasing in $p\in(1,\infty)$.
\end{conjecture}
\section{\bf Some relations for elementary functions}
\bigskip

\begin{lemma}\label{fir} For $x\in(0,1)$,
the following functions
$$f_1(k)=\sin(x^k)^{1/k}\,,f_2(k)=\cos(x^k)^{1/k}\,,
f_4(k)=\tanh(z^k)^{1/k}\,,$$
are increasing in $(0,\infty)$.
\end{lemma}

\begin{proof} We get
$$f^{'}_1(k)=(x^k {\rm cot}(x^k)\log(x^k)-\log(\sin(x^k)))\sin(x^k)^{1/k}/k^2,$$
which is positive because
$$h_1(y)=y \,{\rm cot}(y)\log(y)-\log(\sin(y))\geq 0\,.$$
For $f_2$ we get
$$f^{'}_2(k)=-(x^k \tan(x^k)\log(x^k)+\log(\cos(x^k)))\cos(x^k)^{1/k}/k^2,$$
which is positive because the function
$h_2(y)=y \tan(y)\log(y)+\log(\cos(y))\leq  0$.
 For $f_3$ we get
 $$f^{'}_3(k)=\frac{{\rm tanh}(z^k)^{1/k}}{k^2}
 (2z^k\log(z^k)/\sinh(2z^k)-\log(\tanh(z^k))).$$
 Let $$h_3(y)=2y\log(y)/\sinh(2y)-\log(\tanh(y)),\,y=z^k\in(0,\infty).$$
 Clearly $h_3(y)>0$ for $y>1$. For $y\in(0,1)$ we see that $h_3(y)>0$ iff
 $$\frac{2y}{\sinh(2y)}\frac{\log(y)}{\log(\tanh(y))}\leq 1$$
 which holds because $y>\tanh(y)$. In conclusion, $f^{'}_3(k)>0$ for all $z\in(0,\infty)$.


\end{proof}

\begin{lemma}\label{k5} The following inequalities hold
\begin{enumerate}

\item $\sqrt{{\rm arccos}(r^2){\rm arccos}(s^2)}< {\rm arccos}(r\,s)\,,r,s\in(0,1)$\\

\item ${\rm arctan}(r){\rm arctan}(s)<\sqrt{{\rm arctan}(r^2){\rm arctan}
(s^2)}< {\rm arctan}(r\,s)\,,$\\
for $r,s\in(0,1)$

\item $\sqrt{{\rm arcosh}(r^2)\,{\rm arcosh}(s^2)}<
 {\rm arcosh}(r\,s);\;\,r,s\in(1,\infty)\,.$\\
\end{enumerate}
\end{lemma}

\begin{proof}
For (1) we let $f(x)=\log({\rm arccos}(e^{-x}))\;,x>0$, and get
$$f^{''}(x)=-\frac{\sqrt{e^{2x}-1}+e^{2x}{\rm arccos}(e^{-x})}
{(e^{2x}-1)^{3/2}{\rm arccos}^2(e^{-x})}\leq 0,$$
hence $f$ is concave, and the inequality follows.

For (2) we define $g(x)=\log({\rm arcsin}(e^{-x}))\;,x>0$ and obtain
$$g^{''}(x)=\frac{e^x \left(\left(e^{2 x}-1\right) \tan
   ^{-1}\left(e^{-x}\right)-e^x\right)}{\left(e^{2 x}+1\right)^2 \tan
   ^{-1}\left(e^{-x}\right)^2}<0\,,$$
because $y<\tan(y/(1-y^2))$ for $y\in(0,1)$, hence $g$ is concave. Therefore
the first inequality of (2) follows and the second one follows
from Lemma \ref{fir}. Finally we define
$h(x)=\log({\rm arcosh}(e^{x}))\,,\;x>0$ and get
$$h^{''}(x)=-\frac{e^x(e^x\sqrt{e^{2x}-1}-{\rm arcosh}
(e^{x}))}{(e^{2x}-1)^{3/2}\,{\rm arcosh}^2(e^{x})}<0\,.$$
This implies the proof of (3).
 \end{proof}

\begin{lemma}\label{l} For $r,s\in(0,\infty)$, we have
\begin{enumerate}
\item $\cosh(r\,s)< \sqrt{\cosh(r^2)\cosh(s^2)}< \cosh(r)\cosh(s)\,,$\\
here second inequality holds for $r,s\in(0,1)$,

\item $\tanh(r)\tanh(s)<\sqrt{\tanh(r^2)\tanh(s^2)}
<\sqrt{\tanh(r^2\,s^2)}\,.$
\end{enumerate}
\end{lemma}

\begin{proof}
For (1) we let
$g_1(x)=\log(\cosh(e^{-x}))$ and $g_2(x)=\log(\cosh(e^{x}))\,,\;x>0$, and we get
$$g_1^{''}(x)=e^{-2x}(1/(\cosh^2(e^{-x}))+e^x\tanh(e^{-x}))>0,$$
$$g_2^{''}(x)=e^{x}(e^x/(\cosh^2(e^{x}))+\tanh(e^{x}))>0,$$
hence $g_1$ and $g_2$ are convex, and the first inequality of (1)
holds, and its second inequality follows from Lemma \ref{fir}.
The firstinequality of (2) follows from Lemma \ref{fir}. For the second one
let $h_1(x)=\log(\tanh(e^{-x})),\,x>0$ and get
$$e^{-2 x} \left(-\text{csch}^2\left(e^{-x}\right)+2 e^x
   \text{csch}\left(2
   e^{-x}\right)-\text{sech}^2\left(e^{-x}\right)\right)$$
which is negative, hence $h_1$ is concave. Again, let $h_2(x)=\log(\tanh(e^{x}))$ and get
$$-e^x \left(e^x \text{csch}^2\left(e^x\right)-2 \text{csch}\left(2
   e^x\right)+e^x \text{sech}^2\left(e^x\right)\right)<0\,.$$
   This implies that $h_2$ is also concave, and the second inequality of (2) holds for
   $r,s\in(0,\infty)$.

\end{proof}


%
%
%
%
%
%

\begin{lemma}\label{b} For $y\in(0,1)$, we have
\begin{equation}\label{b1}
\frac{\pi}{2}\,y\cot\left(\frac{\pi\,y}{2}\right)\log\,y\leq
\log\left(\sin\left(\frac{\pi\,y}{2}\right)\right)\,,
\end{equation}
\begin{equation}\label{b2}
y\coth\left(y\right)\log\,y\leq
\log\left(\sinh\left(y\right)\right)\,,
\end{equation}

\begin{equation}\label{b3}
\log \left(\tan \left(\frac{\pi  y}{2}\right)\right)\geq\frac{ \pi}{2}  y
   \log (y) \csc \left(\frac{\pi  y}{2}\right) \sec \left(\frac{\pi
   y}{2}\right)\,.
\end{equation}
\end{lemma}

\begin{proof} Let
$f(y)=\frac{\pi}{2}\,y\cot\left(\frac{\pi\,y}{2}\right)\log\,y-
\log\left(\sin\left(\frac{\pi\,y}{2}\right)\right)\,.$
We get
\begin{eqnarray*}
f^{'}(y)&=& \frac{\pi}{2}\cot\left(\frac{\pi\,y}{2}\right)\log\,y-
\frac{1}{4}\,y\pi^2\csc^2\left(\frac{\pi\,y}{2}\right)\log \,y\\
&=& \frac{\pi}{2}\log(y^{-1})\left(\frac{\pi\,y}{2}\frac{1}{\sin^2(\pi\,y/2)}-
\frac{\cos(\pi\,y/2)}{\sin(\pi\,y/2)}\right)\\
&=& \frac{\pi}{2}\frac{\log(y^{-1})}{\sin^2(\pi\,y/2)}\left(\frac{\pi\,y}{2}-
\sin\left(\frac{\pi\,y}{2}\right)\cos\left(\frac{\pi\,y}{2}\right)\right)\\
&=& \frac{\pi}{2}\frac{\log(y^{-1})}{\sin^2(\pi\,y/2)}\left(\frac{\pi\,y}{2}-
\frac{\sin(\pi\,y)}{2}\right).
\end{eqnarray*}
This is positive because $x\geq\sin\,x$ for $x\in(0,2\pi)$, and $f(1)=0$ and
 this completes the proof. Next, let
$$g(y)=y\coth\left(y\right)\log\,y-
\log\left(\sinh\left(y\right)\right)\,.$$
We get
$$g^{'}(y)=\frac{ \log(1/y)}{\sinh^2(y)}\left(y-
\sinh(y)\cosh(y)\right)\leq 0,$$
because $\sinh x \geq x/\cosh x$ for $x>0$. Moreover, $g$ tends to
zero when $y$ tends to zero and this implies the proof of
(\ref{b2}). Next, let
$$h(y)=\log \left(\tan \left(\frac{\pi  y}{2}\right)\right)-\frac{ \pi}{2}  y
   \log (y) \csc \left(\frac{\pi  y}{2}\right) \sec \left(\frac{\pi
   y}{2}\right).$$
We see that
\begin{eqnarray*}
h^{'}(y)&=&-\frac{\pi ^2}{4}  y \log (y) \sec ^2\left(\frac{\pi
   y}{2}\right)+\frac{1}{4} \pi ^2 y \log (y) \csc ^2\left(\frac{\pi
   y}{2}\right)\\
   & &-\frac{\pi}{2}   \log (y) \csc \left(\frac{\pi
   y}{2}\right) \sec \left(\frac{\pi  y}{2}\right)\\
&=&\pi  \log \left(\frac{1}{y}\right) \csc ^2(\pi  y) (\sin (\pi  y)-\pi  y
   \cos (\pi  y))\leq 0,
\end{eqnarray*}
because $x\leq \tan x$ for $x\in(0,1)$. Hence $h$ is increasing and tends to $\log(\pi/2)$ when
$y$ tends to zero and this implies the proof.
\end{proof}

\begin{lemma}\label{a0a}
\begin{enumerate}
\item The function
$$H(y)=\frac{1}{2} \pi  \log \left(\frac{1}{y^y}\right) \cot \left(\frac{\pi  y}
{2}\right)-\log \left(\csc \left(\frac{\pi  y}{2}\right)\right)\,$$
is decreasing from $(0,1)$ onto $(0,\log(\pi/2))$.\\
\item The function
$$G(y)=\log \left(\cosh \left(\frac{\pi  y}{2}\right)\right)-\frac{1}{2} \pi  y
   \log (y) \tanh \left(\frac{\pi  y}{2}\right)$$
is increasing from $(0,1)$ onto $(0,\pi\log(\cosh(\pi/2))/2)$.
\end{enumerate}
\end{lemma}

\begin{proof} We get
\begin{eqnarray*}
H^{'}(y)&=&-\frac{ \pi }{4} \csc ^2\left(\frac{\pi  y}{2}\right)
\left(\pi  \log \left(y^{-y}\right)+\log (y) \sin (\pi  y)\right)\\
&=&-\frac{\pi}{4}   \csc ^2\left(\frac{\pi  y}{2}\right)
\left(\pi \,y \log (1/y)- \sin (\pi  y)\log (1/y)\right),\\
\end{eqnarray*}
which is positive. Next,
$$G^{'}(y)=-\frac{1}{2} \pi  \log (y) \tanh \left(\frac{\pi
   y}{2}\right)-\frac{1}{4} \pi ^2 y \log (y)
   \text{sech}^2\left(\frac{\pi  y}{2}\right)>0\,,$$
    and the limiting values follow easily.
\end{proof}

\begin{lemma}\label{aa} The following function is increasing from
$(0,1)$ onto $(0,\pi(\log(\pi/2))/2)$
$$g(x)=\frac{x}{\sqrt{1-x^2}}\log\left(\frac{1}{x}\right)-{\rm arcsin}(x)
\log\left(\frac{1}{{\rm arcsin}(x)}\right)\,.$$
In particular,
$$x^{x/\sqrt{1-x^2}}< {\rm arcsin}(x)^{{\rm arcsin}(x)}
 < \left(\frac{\pi}{2}\right)^{\pi/2}\, x^{x/\sqrt{1-x^2}}\,.$$
\end{lemma}

\begin{proof} We get
\begin{eqnarray*}
g^{'}(x)&=&-\frac{x^2 \log (x)}{\left(1-x^2\right)^{3/2}}-\frac{\log
   (x)}{\sqrt{1-x^2}}+\frac{\log \left({\rm arcsin}(x)\right)}{\sqrt{1-x^2}}\\
   &=&\frac{\log(1/x)-(1-x^2)\log(1/{\rm arcsin}(x))}{(1-x^2)^{3/2}}\\
   &=&\frac{\log({\rm arcsin}(x)^{(1-x^2)}/x)}{(1-x^2)^{3/2}}\,,
\end{eqnarray*}
which is clearly positive, and $g$ tends to zero when $x$ tends to zero and $1$.
\end{proof}

\begin{lemma}\label{aaa} For $x\in(0,1)\,,$  the following functions
$$f(k)=\sin\left(\frac{\pi}{2}\,x^k \right)^{1/k},\,\,
g(k)=\tan\left(\frac{\pi}{2}\,x^k \right)^{1/k},\,\,
h(k)=\sinh\left(x^k\right)^{1/k}\,,$$
are decreasing in $(0,\infty)$. In particular, for $k\geq 1$
$$\sqrt[k]{\sin\left(\frac{\pi}{2}\,x^k \right)}\leq \sin\left(\frac{\pi}{2}\,x\right)\leq
\sin\left(\frac{\pi}{2}\sqrt[k]{x}\right)^k,$$
$$\sqrt[k]{\tan\left(\frac{\pi}{2}\,x^k \right)}\leq \tan\left(\frac{\pi}{2}\,x\right)\leq
\tan\left(\frac{\pi}{2}\sqrt[k]{x}\right)^k,$$
$$\sqrt[k]{\sinh\left(x^k\right)}\leq \sinh\left(x\right)\leq
\sinh\left(\sqrt[k]{x}\right)^k.$$
\end{lemma}

\begin{proof} We get
\begin{eqnarray*}
f^{'}(k)&=& \sqrt[k]{\sin \left(\frac{\pi  x^k}{2}\right)}
 \left(\frac{\pi  x^k \log (x) \cot \left(\frac{\pi  x^k}{2}\right)}{2
   k}-\frac{\log \left(\sin \left(\frac{\pi  x^k}{2}\right)\right)}{k^2}\right)\\
&=&-\frac{1}{2\,k^2}\sqrt[k]{\sin \left(\frac{\pi  x^k}{2}\right)}
\left(\pi k \,x^k \log (1/x) \cot \left(\frac{\pi  x^k}{2}\right)-2 \log \left(1/\sin
   \left(\frac{\pi  x^k}{2}\right)\right)\right),\\
\end{eqnarray*}
which is negative by Lemma \ref{a0a}(1). Next, we get
$$
g^{'}(k)=\sqrt[k]{\tan \left(\frac{\pi  x^k}{2}\right)} \left(\frac{\pi  x^k \log
   (x) \csc \left(\frac{\pi  x^k}{2}\right) \sec \left(\frac{\pi
   x^k}{2}\right)}{2 k}-\frac{\log \left(\tan \left(\frac{\pi
   x^k}{2}\right)\right)}{k^2}\right)\leq 0,
$$
by (\ref{b3}). Finally,
\begin{eqnarray*}
h^{'}(k)&=&\sqrt[k]{\sinh \left(x^k\right)} \left(\frac{x^k \log (x) \coth
   \left(x^k\right)}{k}-\frac{\log \left(\sinh
   \left(x^k\right)\right)}{k^2}\right)\\
&=&\sqrt[k]{\sinh \left(x^k\right)} \left(x^k \log (x^k) \coth
   \left(x^k\right)-\log \left(\sinh \left(x^k\right)\right)\right)(1/k^2),\\
\end{eqnarray*}
which is negative by inequality (\ref{b2}), and this completes the proof.
\end{proof}

%

\begin{lemma}\label{cos1} The following functions
$$f(k)=\cos\left(\frac{\pi}{2}x^{1/k}\right)^k\,, x\in (0,1)\,,$$
$$g(k)=\cosh\left(x^k\right)^{1/k},\,\, x\in (0,1)\,,$$
$$h(k)={\rm arcosh}\left(\frac{\pi}{2}
\,x^{k}\right)^{1/k},\,\, x\in (1,\infty)\,,$$
are decreasing in $(0,\infty)$. In particular, for $k\geq 1$
$$\cos\left(\frac{\pi}{2}\sqrt[k]{x}\right)^k\leq\cos\left(\frac{\pi}{2}
\,x\right)\leq\sqrt[k]{\cos\left(\frac{\pi}{2}\,x^k\right)}
\,,$$
$$\sqrt[k]{\cosh\left(x^k\right)}\leq\cosh\left(
x\right)\leq\cosh\left(\sqrt[k]{x}\right)^k
\,,$$
$$\sqrt[k]{{\rm arcosh}\left(\frac{\pi}{2}\,x^k\right)}\leq{\rm arcosh}
\left(\frac{\pi}{2}\,x\right)\leq{\rm arcosh}\left(\frac{\pi}{2}\sqrt[k]{x}\right)^k
.$$
\end{lemma}

\begin{proof}
We get
$$f^{'}(x)=\cos ^k\left(\frac{1}{2} \pi  x^{1/k}\right) \left(\frac{\pi
   x^{1/k} \log (x) \tan \left(\pi
   x^{1/k}/2\right)}{2 k}+\log \left(\cos \left(\pi
   x^{1/k}/2\right)\right)\right)\leq 0\,$$
and proof of $g$ follows from Lemma \ref{k5}(1).

Finally, for $y\geq \pi/2$, let
$$j(y)={\rm arcosh}\left(
y\right)\log\left({\rm arcosh}\left(y
\right)\right)-\frac{y\log(2y/\pi)}{\sqrt{y^2-1}},$$
and $$j^{'}(y)=\frac{\log \left(2 y/\pi \right)}
{\left(y^2-1\right)^{3/2}}+\frac{\log \left({\rm arcosh}
   (y)\right)}{\sqrt{y^2-1}}> 0,$$
   and $$j(\pi/2)={\rm arcosh}(\pi/2)\log({\rm arcosh}(\pi/2))\equiv 0.0235\,.$$
With $z=x^{k}$ we get
$$h^{'}(x)=\frac{{\rm arcosh}\left(\pi
z /2\right)^{1/k}}{k^2{\rm arcosh}\left(\pi
z/2\right)}\left(\frac{\pi z\log(z)}{2\sqrt{(\pi z/2)^2-1}}-
{\rm arcosh}\left(\frac{\pi}{2}
z/2\right)\log\left({\rm arcosh}\left(\frac{\pi }{2}
\right)\right)\right)\,.$$
This is negative, because $j(y)>0$ for $y>\pi/2$.
\end{proof}

%
%
%
%


\begin{lemma}\label{k1} The following relations hold
\begin{enumerate}
\item $\sin(r)\sin(s)<\sqrt{\sin (r^2)\sin (s^2)},\,r,s\in(0,1)$\,,\\

\item $\cos(r)\cos(s)<\sqrt{\cos (r^2)\cos (s^2)}
< \cos(r\,s)\,,$

\item $\tan(r)\tan(s)>
\sqrt{\tan (r^2)\tan (s^2)}> \tan(r\,s)$,\\
the first inequalities in (2) and (3) hold for $r,s\in(0,\sqrt{\pi/2})$,
and second ones for $r,s\in(0,1)$.
\end{enumerate}
\end{lemma}

\begin{proof} Clearly (1) and the fist inequality of (2) follwos
from Lemmas \ref{fir}and \ref{cos1},
respectively.
Let $g(x)=\log(\cos(\pi\,e^{-x}/2)),\, x>0$, we get
\begin{eqnarray*}
g{''}(x)&=&-\frac{\pi ^2}{4} \,e^{-2 x}  \sec ^2\left(\frac{e^{-x} \pi
   }{2}\right)-\frac{\pi}{2}\, e^{-x}   \tan \left(\frac{e^{-x} \pi
   }{2}\right)\\
   &=&-\frac{\pi}{4} \,e^{-2 x}   \sec ^2\left(\frac{e^{-x} \pi
   }{2}\right) \left(e^x \sin \left(e^{-x} \pi \right)+\pi \right)\leq 0,
\end{eqnarray*}
   and the second inequality of (2) follows.

    For (3), we define
   $h(x)=\log(\tan(\pi\,e^{-x}/2)),\, x>0$, and we get
 $$
h{''}(x)=e^{-x} \pi  \left(1-e^{-x} \pi  \cot \left(e^{-x} \pi
   \right)\right) \csc \left(e^{-x} \pi \right)\geq 0,$$
hence $h$ is convex, and the second inequality follows easily, and the
first one follows from
Lemma \ref{aaa}.
\end{proof}

\begin{lemma}\label{c} For a fixed $x\in (0,1)$, the function
$g(k)=(\cos\,kx+\sin kx)^{1/k}$
is decreasing in $(0,1)$.
\end{lemma}

\begin{proof}
Differentiation yields
$$g^{'}(k)=(\sin (k x)+\cos (k x))^{\frac{1}{k}} \left(\frac{k x (\cos (k x)-\sin
   (k x))}{\sin (k x)+\cos (k x)}-\log (\sin (k x)+\cos (k x))\right)/k^2\,.$$
To prove that this is positive, we let
$z=k\,x\,,y=\cos\,z+\sin\,z\leq 1.1442$\,
$$h(z)=(\cos\,z+\sin\,z)\log(\cos\,z+\sin\,z)-z(\cos\,z-\sin\,z)\,,$$
and observe that
\begin{eqnarray*}
h^{'}(z)&=&z\,\cos\,z+(\cos\,z-\sin\,z)\log(\cos\,z+\sin\,z)+z\,\sin\,z\\
&=&zy+\log\,y^{\cos\,z}-\log\,y^{\sin\,z}\geq 0,
\end{eqnarray*}
because $e^{zy}>y^{\sin\,z}$. This implies that $g^{'}(k)\geq 0$.
\end{proof}
\section{\bf Appendix}

In the following tables we give the values of $p$-analogue functions for some specific values of its domain with $p=3$ computed with Mathematica$^\registered\,.$ For instance,
we can define \cite{ru}

\bigskip
{\tt

arcsinp[p\_, x\_] := x *Hypergeometric2F1[1/p, 1/p, 1 + 1/p, x\^{ }p]

sinp[p\_, y\_] := x /. FindRoot[ arcsinp[p, x] == y, \{x, 0.5 \}]

}

\bigskip

\begin{displaymath}
\begin{array}{|c|c|c|c|c|c|}
\hline
x&{\rm arcsin_p}(x)&{\rm arccos_p}(x)&{\rm arctan_p}(x)&{\rm arsinh_p}(x)&{\rm artanh_p}(x)\\
\hline

0.00000 & 0.00000 & 1.20920 & 0.00000 & 0.00000&0.00000 \\

0.25000&0.25033&1.17782&0.24903&0.24968&0.25099\\

0.50000&0.50547&1.07974&0.48540&0.49502&0.51685\\

0.75000&0.78196&0.88660&0.68570&0.72710&0.85661\\

1.00000&1.20920&0.00000&0.83565&0.93771&\infty\\

\hline
\end{array}
\end{displaymath}

\bigskip

\begin{displaymath}
\begin{array}{|c|c|c|c|c|c|}
\hline
x&{\rm sin_p}(x)&{\rm cos_p}(x)&{\rm tan_p}(x)&{\rm sinh_p}(x)&{\rm tanh_p}(x)\\
\hline

0.00000 & 0.00000 & 1.00000 & 0.00000 & 0.00000&0.00000 \\

0.25000&0.24967&0.99478&0.25098&0.25033&0.24903\\

0.50000&0.49476&0.95788&0.51652&0.50518&0.48517\\

0.75000&0.72304&0.85362&0.84704&0.77588&0.68283\\

1.00000&0.91139&0.62399&1.46058&1.08009&0.82304\\

\hline
\end{array}
\end{displaymath}

\bigskip

With a normalization different from ours, some eigenvalue problems of the
$p$-Laplacian have been studied in \cite{br}.

\end{document}